\documentclass{amsart}
\usepackage[T1]{fontenc}
\usepackage[utf8]{inputenc}
\usepackage{amsmath}
\usepackage{array}
\usepackage{amsthm}
\usepackage{amssymb}
\input xy
\xyoption{all}

\usepackage[cmtip,all]{xy}

\usepackage{stmaryrd}

\usepackage[english]{babel}

\usepackage{tikz}
\usetikzlibrary{matrix,arrows,decorations.pathmorphing}
\usepackage{tikz-cd}

\usepackage{enumerate}

\usepackage{babel}

\usepackage{xcolor}

\newtheorem{thmintro}{Theorem}
\newtheorem{corintro}[thmintro]{Corollary}
\newtheorem{printro}[thmintro]{Proposition}
\newtheorem{pbintro}[thmintro]{Problem}

\newtheorem{thm}{Theorem}[section]
\newtheorem{pr}[thm]{Proposition}
\newtheorem{cor}[thm]{Corollary}
\newtheorem{lm}[thm]{Lemma}

\theoremstyle{definition}
\newtheorem{defi}[thm]{Definition}
\newtheorem{defi-prop}[thm]{Proposition-Definition}
\newtheorem{ex}[thm]{Example}
\newtheorem{applicintro}[thmintro]{Application}

\theoremstyle{remark}
\newtheorem{rk}[thm]{Remark}

\newcommand{\A}{{\mathcal{A}}}
\newcommand{\B}{{\mathcal{B}}}
\newcommand{\C}{{\mathcal{C}}}
\newcommand{\D}{{\mathcal{D}}}
\newcommand{\I}{{\mathcal{I}}}
\newcommand{\M}{{\mathcal{M}}}
\newcommand{\red}{\mathrm{red}}
\newcommand{\J}{{\mathcal{J}}}
\newcommand{\Md}{\text{-}\mathbf{Mod}}
\newcommand{\Mdd}{\mathbf{Mod}\text{-}}
\newcommand{\id}{\mathrm{id}}
\newcommand{\res}{{\mathrm{res}}}
\newcommand{\op}{{\mathrm{op}}}
\newcommand{\Fp}{{\mathbb{F}_p}}
\newcommand{\Proj}{\mathbf{P}}
\newcommand{\GL}{\operatorname{GL}}
\newcommand{\Ext}{\mathrm{Ext}}
\newcommand{\Tor}{\mathrm{Tor}}
\newcommand{\Hom}{\mathrm{Hom}}
\newcommand{\Cr}{\mathrm{cr}}

\begin{document}

\title{Separation and excision in functor homology}

\author[A. Djament]{Aur\'elien Djament}
\address{CNRS, UMR 7539, Laboratoire Analyse, G\'eom\'etrie et Applications, LAGA, Universit\'e
Sorbonne Paris Nord, F-93430, Villetaneuse, France}
\email{djament@math.cnrs.fr}
\urladdr{https://djament.perso.math.cnrs.fr/}

\author[A. Touz\'e]{Antoine Touz\'e}
\address{Univ. Lille, CNRS, UMR 8524 - Laboratoire Paul Painlev\'e, F-59000 Lille, France}
\email{antoine.touze@univ-lille.fr}
\urladdr{https://pro.univ-lille.fr/antoine-touze/}
\thanks{This author is partly supported by the Labex CEMPI (ANR-11-LABX-0007-01)}

\subjclass[2020]{Primary 18A25, 18G15; Secondary 18E05, 18G31}

\keywords{Functor homology; polynomial functor}

\begin{abstract}
We prove separation and excision results in functor homology. These results explain how the global Steinberg decomposition of functors proved by Djament, Touzé and Vespa \cite{DTV} behaves in Ext and Tor computations. 
\end{abstract}

\maketitle

\selectlanguage{french}
\renewcommand{\abstractname}{R\'esum\'e}
\begin{abstract}
Nous démontrons des résultats de séparation et d'excision en homologie des foncteurs. Ces résultats expliquent comment la décomposition de Steinberg globale des foncteurs démontrée par Djament, Touzé et Vespa  \cite{DTV} se comporte dans les calculs d'Ext et de Tor. 
\end{abstract}

\selectlanguage{english}

\section{Introduction}
If $k$ is a commutative ring and $\C$ is a category, we let $k[\C]\Md$ stand for the abelian category of functors $\C\to k\Md$, and natural transformations between them. 
These categories are classical objects of representation theory, whose study goes back at least to Mitchell's foundational article \cite{Mi72}. If $\C=\Proj_R$ is the category of finitely generated projective right modules over a ring $R$, homological algebra in $k[\Proj_R]\Md$, is also notoriously related to invariants of $K$-theoretic nature such as the topological Hochschild homology of $R$ \cite{PiraWald} or the homology of $\GL(R)$ and its classical subgroups \cite{Sco,DjaR}.

Objects of $k[\C]\Md$ can be very complicated, but if $\C=\A$ is an additive category, the structure is greatly clarified by considering two classes of functors.
The first one is the class of \emph{polynomial functors}, initially introduced by Eilenberg and MacLane \cite{EML}. Such functors are a natural generalization of additive functors; a typical example of a polynomial functor $T$ of degree $d$ can be constructed from $d$ additive functors $A_i$ by taking tensor products over $k$:
$T(x)= A_1(x)\otimes\cdots\otimes A_d(x)$. 
The second class of functors is the class of \emph{antipolynomial functors}, introduced in \cite{DTV}. Namely, an antipolynomial functor is a functor which factors through an additive category $B$ which has finite $\Hom$-sets, of cardinal invertible in $k$, cf. definition \ref{defi-A}. As the name suggests, the structural properties of antipolynomial and polynomial functors are essentially `orthogonal' and the only functors which are both polynomial and antipolynomial are the constant functors. 

As the global Steinberg decomposition of \cite[Thm 2]{DTV} shows it, a huge number of interesting objects of $k[\A]\Md$ can be created by mixing together these two classes of functors. To be more specific, a functor $B:\A\times\A\to k\Md$ is of \emph{antipolynomial-polynomial-type} (AP-type) if it is antipolynomial as a functor of its first variable and polynomial as a functor of the second variable, cf. definition \ref{defi-AP}. Such a bifunctor of AP-type yields an object $\Delta^*B$ of $k[\A]\Md$ defined by $(\Delta^*B)(x)=B(x,x)$.
The next theorem is the first main result of this article. It shows that this construction has a very rigid behavior from the homological point of view. We denote by $\Ext^*_{k[\C]}(F,G)$ the Ext between two objects of $k[\C]\Md$, and similarly for $\Tor$ (the first argument of the $\Tor$ must be an object of $\Mdd k[\C]$, that is, a contravariant functor from $\C$ to $k$-modules, see section \ref{sec-recoll} for more details on $\Ext$ and $\Tor$).
\begin{thmintro}[separation]\label{thm-AP-type}
Let $B$ and $C$ be bifunctors of AP-type. The natural graded morphisms
\begin{align*}
&\Ext^*_{k[\A\times\A]}(B,C)\xrightarrow{\simeq}\Ext^*_{k[\A]}(\Delta^*B,\Delta^*C)\;,\\
&\Tor_*^{k[\A]}(\Delta^*B,\Delta^*C)\xrightarrow{\simeq} \Tor_*^{k[\A\times \A]}(B,C)
\end{align*}
are isomorphisms.
\end{thmintro}
Our separation theorem essentially says that there is no homological interaction between the antipolynomial and the polynomial part of a functor. This property is probably best understood when the bifunctors in play are constructed from tensor products. Given two functors $F$ and $G$ from $\A$ to $k$-modules, we denote by $F\otimes G$ their objectwise tensor product over $k$, that is $(F\otimes G)(x)=F(x)\otimes G(x)$. The separation theorem has the following consequence.
\begin{corintro}[Separation for tensor products]Assume that $k$ is a field, that $A_1$ and $A_2$ are antipolynomial functors and that $P_1$ and $P_2$ are polynomial functors. There is a graded isomorphism:
$$\Tor_*^{k[\A]}(A_1\otimes P_1,A_2\otimes P_2)\simeq \Tor_*^{k[\A]}(A_1,A_2)\otimes \Tor_*^{k[\A]}(P_1,P_2)\;.$$
If in addition $A_1$ and $P_1$ are of type $\mathrm{fp}_\infty$ there is also a graded isomorphism:
$$\Ext^*_{k[\A]}(A_1\otimes P_1,A_2\otimes P_2)\simeq \Ext^*_{k[\A]}(A_1, A_2)\otimes \Ext^*_{k[\A]}(P_1, P_2)\;.$$
\end{corintro}
\begin{proof}
Set $B(x,y)=A_1(x)\otimes P_1(y)$ and $C(x,y)=A_2(x)\otimes P_2(y)$. The separation theorem shows that $\Ext$ and $\Tor$ between $A_1\otimes P_1$ and $A_2\otimes P_2$ are isomorphic to $\Ext$ and $\Tor$ between $B$ and $C$. The latter can be computed with the classical K\"unneth formula, see for example \cite[Lm 8.2]{DT-Add} for the $\Tor$ version and proposition \ref{pr-Kunneth} for the $\Ext$ version.
\end{proof}

\begin{rk}[The interest of tensor products]
It is proved in \cite[Thm 4.12]{DTV} that if $k$ is a field containing all the roots of unity, then every simple functor with finite dimensional values is a tensor product $A\otimes P$ where $A$ is a simple antipolynomial functor and $P$ is a simple polynomial functor.
\end{rk}

The separation theorem reduces the understanding of $\Ext$ and $\Tor$ between many functors of interest to the following general problem.
\begin{pbintro}\label{pbintro}
Understand $\Ext$ and $\Tor$ between two functors of the same type, that is, between two polynomial functors or between two antipolynomial functors. 
\end{pbintro}

A great deal of work has been done over the past thirty years to understand problem  \ref{pbintro}, at least in its polynomial part, and rather satisfying answers are known when the category $\A$ is not "too complicated". For example, if $\A$ is the category of finite dimensional vector spaces over a finite field and $k$ is a field of the same characteristic, the work of Franjou Friedlander Scorichenko Suslin \cite{FFSS} has shown that in many cases $\Ext$ between polynomial functors can be computed in terms of $\Ext$ between "Frobenius twisted" strict polynomial functors -- or equivalently \cite[Thm 3.2]{FS} in terms of $\Ext$ between "Frobenius twisted" modules over classical Schur algebras. The latter $\Ext$ are now well-understood, see the survey \cite{Touze-Survey} for the "untwisting formulas" and connections to other problems.    
The results of \cite{FFSS} have been recently generalized \cite{DT-Add,DT-Documenta} when $\A$ is a linear category over some field and $k$ is a field of the same characteristic, but we are still far from understanding the case of an arbitrary additive category $\A$. Similarly, the homological properties of antipolynomial functors  are relatively well-understood when $\A$ is the category of finite dimensional vector spaces over a finite field and $k$ is a field of different characteristic by the work of Kuhn \cite{Ku-adv}, but \textit{very little} is understood over an arbitrary additive category $\A$, because there is a big gap between the situation studied by Kuhn \cite{Ku-adv} and the structure of general antipolynomial functor categories, much more intricate (see \cite{DG}). 

The next main result of the present article is an excision theorem, namely theorem \ref{thm-intro-2}. This excision theorem brings the study of problem \ref{pbintro} for general $\A$ closer to the cases where it is currently understood, by showing that in some cases the category $\A$ can be replaced by a simpler quotient category $\B$ without altering $\Ext$ and $\Tor$. It is particularly useful for antipolynomial functors (see application~\ref{exintro1}).

To be more specific, let $\phi:\A\to \B$ be an additive quotient, that is, $\phi$ is a full and essentially surjective additive functor between two additive categories. Then every object of $F$ of $k[\B]\Md$ yields an object $\phi^*F$ of $k[\A]\Md$ defined by $(\phi^*F)(x)=F(\phi(x))$. An elementary argument shows that composition with $\phi$ yields an isomorphism:
$$\Hom_{k[\B]}(F,G)\simeq \Hom_{k[\A]}(\phi^*F,\phi^*G)\;.$$
Our first excision result shows that under an elementary hypothesis on $\B$, such an isomorphism extends to higher $\Ext$ (and also to $\Tor$). 
\begin{thmintro}[excision]\label{thm-intro-2}
Let $\phi:\A\to \B$ be an additive quotient, such that $\B(x,x)\otimes_{\mathbb{Z}} k=0$ for all objects $x$ of $\B$. Then for all functors $F$ and $G$ defined on $\B$, the natural graded morphisms 
\begin{align*}
&\Ext^*_{k[\B]}(F,G)\xrightarrow{\simeq}\Ext^*_{k[\A]}(\phi^*F,\phi^*G)\;,\\
&\Tor_*^{k[\A]}(\phi^*F,\phi^*G)\xrightarrow{\simeq} 
\Tor_*^{k[\B]}(F,G)
\end{align*}
are isomorphisms.
\end{thmintro}

Theorem \ref{thm-intro-2} is obtained as a special case of a more general excision result that we state in theorem \ref{thm-magique-general}. The latter replaces the condition on $\B$ by a weaker (but more technical) condition, which generalizes the `$H$-unital' condition of Suslin and Wodzicki governing  excision in rational algebraic $K$-theory \cite{Suslin-Wodz}. We refer the reader to remark \ref{rk-excis} for more details regarding the link with excision in $K$-theory.

The typical way to use theorem \ref{thm-intro-2} in practice is explained in the following application to antipolynomial functors.
\begin{applicintro}\label{exintro1}
Let $A_1$ and $A_2$ be two antipolynomial functors of $k[\A]\Md$. 
Then there is an additive quotient $\phi:\A\to \B$ such that: 
\begin{enumerate}
\item[(i)] $\B$ is a category with finite $\Hom$ of cardinal invertible in $k$, 
\item[(ii)] there are objects $F_1$ and $F_2$ of $k[\B]\Md$ such that $\A_i=\phi^*F_i$ for $i=1,2$. 
\end{enumerate}
Indeed, each $A_i$ is antipolynomial, hence factors through an additive functor $\phi_i: \A\to \B_i$ whose codomain is a category with finite $\Hom$ of cardinal invertible in $k$. Let $\I_i$ denote the ideal of the morphisms of $\A$ cancelled by $\phi_i$ . Then $A_1$ and $A_2$ both factor through the additive quotient $\phi:\A\to \B=\A/(\I_1\cap\I_2)$, which satisfies (i). Theorem \ref{thm-intro-2} then shows the graded isomorphism:
$$\Ext^*_{k[\A]}(A_1,A_2)\simeq \Ext^*_{k[\B]}(F_1,F_2)$$ 
and a similar isomorphism for $\Tor$. 
Such isomorphisms are interesting because the computations in $k[\B]\Md$ are expected to be simpler than the computations in $k[\A]\Md$ -- for example the fact that $\B$ is $\Hom$-finite allows to use more combinatorial methods.  
\end{applicintro}

We give an illustration of the approach outlined in application \ref{exintro1} in  proposition \ref{pr-vanish}. Namely, we use theorem \ref{thm-intro-2} to deduce the following vanishing result from the results of Kuhn \cite{Ku-adv}. 
\begin{printro}
Assume that $k$ is a field of characteristic zero, and fix an additive quotient $\phi:\A\to \B=\Proj_R$, where $\Proj_R$ is the category of finitely generated (projective) modules over a finite semi-simple ring $R$.
Then we have 
$$\Ext^i_{k[\A]}(A_1,A_2)=0\;,\qquad \Tor_i^{k[\A]}(A_1,A_2)=0\;,$$
for all positive $i$, as soon as $A_1$ and $A_2$ both factor through $\phi$.
\end{printro}

Theorem \ref{thm-intro-2} cannot be used to study the homological properties of polynomial functors. Indeed, by \cite[prop 2.13]{DTV}, the hypothesis that $\B(x,x)\otimes_{\mathbb{Z}}k=0$ for all $x$ implies that no nonconstant polynomial functor factors through $\B$. In theorem~\ref{thm-poly-excis} we prove a replacement of theorem~\ref{thm-intro-2} for polynomial functors, which relies on a lemma (\ref{lm-vanish-pol}) which is also useful for the separation theorem~\ref{thm-AP-type}. Nevertheless, this polynomial excision theorem is less important than theorem~\ref{thm-intro-2} for antipolynomial functors, because it seems to apply in more restricted situations.

\subsection*{Organization of the paper}
We start by short recollections of functor categories in section \ref{sec-recoll}, in order to fix notations and terminology, and to make the article self-contained. 

Section \ref{sec-3} is devoted to the proof of the excision theorem \ref{thm-magique-general} and its applications, including theorem \ref{thm-intro-2} of the introduction (which corresponds to theorem \ref{thm-magique-csq1}). 

Section \ref{sec-4} recalls some background on polynomial functors, and establishes the polynomial replacement for excision in theorem \ref{thm-poly-excis}. 
Finally, section \ref{sec-sep} builds on the results and notions of section \ref{sec-4} to establish the separation property (theorem \ref{thm-AP-type} of the introduction) in theorem \ref{thm-AP-type-bis}. 

The results of section \ref{sec-3}, as well as some of the results in section \ref{sec-4}, depend on some simplicial techniques and on some classical algebraic topology computations. Since we were unable to find references covering all the results that we needed, we have gathered all these results in an appendix in section \ref{sec-app}.

\subsection*{Some general notations and terminology}
We gather here some notations and terminology which are used throughout the article. 
\begin{enumerate}[$\bullet$]
\item The letter $k$ denotes a commutative ring, unadorned tensor products are taken over $k$. The notation $k[X]$ stands for the free $k$-module on a set $X$. We also say that $k[X]$ is the $k$-linearization of $X$.
\item Categories have a class of objects, and we require that the morphisms between any two objects form a set. We say that a category $\C$ is \emph{essentially small} if the isomorphism classes of objects of $\C$ form a set. We say that a category $\C$ is \emph{additive} if it is a $\mathbb{Z}$-category with a zero object and finite biproducts \cite[VIII.2]{ML}.
\item The letter $\A$ denotes an essentially  small additive category.
\item $\Proj_R$ denotes the category of finitely generated right $R$-modules over a ring $R$.
\item An \emph{additive quotient} is a functor $\phi:\A\to \B$ between two additive categories, which is full and essentially surjective.
\end{enumerate}

\section{Recollections of functor categories}\label{sec-recoll}
Foundations of homological algebra in functor categories are studied into details in \cite{Mi72}. In this section, we only recall the notations and the properties which will be useful to us.

\subsection{Abelian structure, monoidal structure, and $\Ext$}
Let $\C$ be a small category let $k$ be a commutative ring. 
We denote by $k[\C]\Md$ the category whose objects are the functors from $\C$ to $k$-modules and whose morphisms are the natural transformations (with the usual composition of natural transformations). This is an abelian  $k$-linear category with enough injectives and projectives, equipped with a symmetric monoidal product $\otimes$ induced by the tensor product over $k$, that is 
$(F\otimes G)(x)=F(x)\otimes G(x)$.
\begin{rk}
The notation $k[\C]\Md$ is inspired by the notations for modules over a ring, and $k[\C]$ recalls the form of the projective objects in functor categories, see section \ref{subsec-std} below, and see section \ref{subsec-add} for analogue of this notation for categories of additive functors.
\end{rk}

We denote by $\C(c,d)$ the $\Hom$-sets in $\C$ while we will use $\Hom_{k[\C]}(F,G)$ to denote the set of morphisms between two functors $F$ and $G$.
We denote by $\Mdd k[\C]$ the category of all functors from $\C^\op$ to $k$-modules. In other words:
$$\Mdd k[\C]=k[\C^\op]\Md\;.$$
Following the notation of for modules over rings, we also denote by $\Hom_{k[\C]}(F,G)$ the $k$-modules of homomorphisms in $\Mdd k[\C]$ (it will always be clear from the context if we work with left or right modules). Composition with a functor $\phi:\C\to \D$ yields restriction functors (denoted in the same way for left and right modules):
$$\phi^*:k[\D]\Md \to k[\C]\Md\;,\qquad \phi^*:\Mdd k[\D]\to \Mdd k[\C]\;.$$
These restriction functors are exact, and they hence induce natural transformations of $\delta$-functors
$$\res^\phi:\Ext^i_{k[\C]}(F,G)\to \Ext^i_{k[\C]}(\phi^*F,\phi^*G)\;.$$

\subsection{Tensor products over $\C$} We denote by $\Mdd k[\C]$ the category of all functors from $\C^\op$ to $k$-modules. In other words 
$$\Mdd k[\C]=k[\C^\op]\Md\;.$$
The tensor product
$$-{\otimes}_{{k[\C]}}-: \Mdd k[\C]\times k[\C]\Md\to k\Md$$ 
is defined by the coend formula $F\otimes_{k[\C]}G = \int^{c\in\C} F(c)\otimes G(c)$.
It is right exact with respect to each variable, and its derived functors are denoted by $\Tor_*^{k[\C]}(F,G)$. Every functor $\phi:\C\to \D$ induces a natural transformation of $\delta$-functors:
$$\res_\phi:\Tor_i^{k[\C]}(\phi^*F,\phi^*G)\to \Tor_i^{k[\D]}(F,G)\;.$$

\subsection{Duality} For all $k$-modules $M$ and for all functors $F:\C^\op\to k\Md$, we denote by $D_MF:\C\to k\Md$ the functor defined by
$$D_MF(c)=\Hom_k(F(c),M)\;.$$ 
This construction defines a duality functor $D_M:\Mdd k[\C]\to k[\C]\Md$. We use the same notation for the duality functor $D_M:k[\C]\Md\to \Mdd k[\C]$. There are isomorphisms, natural with respect to $F$, $G$ and $M$  
$$\Hom_{k[\C]}(F,D_M G)\simeq \Hom_k(F\otimes_{k[\C]}G,M)\simeq \Hom_{k[\C]}(G,D_M F)\;.$$
If $M$ is an injective $k$-module, these isomorphisms can be derived, to give isomorphisms of $\delta$-functors between $\Ext$ and $\Tor$. These isomorphisms are compatible with restriction along $\phi:\C\to \D$ in the following sense. 
\begin{pr}\label{pr-memechose}
Let $M$ be an injective $k$-module, and let $\phi:\C\to \D$ be a functor. There is a commutative diagram of morphisms of $\delta$-functors:
$$\begin{tikzcd}
\Ext^i_{k[\C]}(F,D_M G)\ar{d}{\simeq} \ar{rr}{\res^\phi}
&&
\Ext^i_{k[\D]}(\phi^*F,\phi^*D_M G)\ar{d}{\simeq}\\
\Hom_k(\Tor_i^{k[\C]}(F,G),M)\ar{d}{\simeq}\ar{rr}{\Hom_k(\res_\phi,M)}
&&
\Hom_k(\Tor_i^{k[\D]}(\phi^*F,\phi^*G),M)\ar{d}{\simeq} \\ 
\Ext^i_{k[\C]}(G,D_M F)\ar{rr}{\res^\phi}
&&
\Ext^i_{k[\D]}(\phi^*G,\phi^*D_M F)
\end{tikzcd}\;.$$
\end{pr}

\subsection{Standard objects of $k[\C]\Md$}\label{subsec-std}

The \emph{standard projectives} are the functors of the form $P^c =k[\C(c,-)]$, where $c$ is an object of $\C$. Thus $P^c$ maps every object $d$ to the free $k$-module on $\C(c,d)$. We may also denote this functor by $P_\C^c$ if we wish to emphasize the category $\C$. The Yoneda lemma yields an isomorphism, natural with respect to $F$ and $c$:
\begin{align}\Hom_{k[\C]}(P^c,F)\simeq F(c)\;,\label{eq-Yoneda}\end{align}
and there are `dual Yoneda isomorphisms' for tensor products
\begin{align}P^c\otimes_{k[\C]}F\simeq F(c)\;,\quad F\otimes_{k[\C]}P^c\simeq F(c)\;.\label{eq-Yoneda-dual}
\end{align}
Every object of $k[\C]\Md$ has a projective resolution by direct sums of standard projectives. 
The \emph{standard injectives} are the functors of the form $D_MP^c_{\C^\op}$, where $M$ is an injective $k$-module and $c$ is an object of $\C^\op$. Thus $D_MP^c_{\C^\op}(d)$ is naturally isomorphic to the $k$-module of maps from $\C(d,c)$ to $M$. Every object of $k[\C]\Md$ has an injective resolution by products of standard injectives.

\subsection{Additive categories and functors}\label{subsec-add}
If $\A$ is an essentially small additive category, we denote by $k\otimes_\mathbb{Z}\A\Md$ the full subcategory of 
$k[\A] \Md$ consisting of the additive functors. This is an abelian subcategory with enough injectives and projectives. 
The \emph{standard projectives} of $k\otimes_\mathbb{Z}\A\Md$ are the functors $h^a=k\otimes_\mathbb{Z}\A(a,-)$ which send every object $b$ of $\A$ to the $k$-module $k\otimes_\mathbb{Z}\A(a,b)$. We also denote the standard projectives by $h^a_\A$ if we want to emphasize the category $\A$. The Yoneda lemma yields an isomorphism of $k$-modules, natural with respect to $a$ and the additive functor $A$:
$$\Hom_{k\otimes_{\mathbb{Z}}\A}(h^a,A)\simeq A(a)\;,$$
and there are similar natural isomorphisms for the tensor products:
\[h^a\otimes_{k[\A]}A\simeq A(a)\;,\quad A\otimes_{k[\A]}h^a\simeq A(a)\;.\]
Every additive functor has a projective resolution by a direct sum of standard projectives. We will denote by $\Mdd k\otimes_{\mathbb{Z}}\A$ the full subcategory of $\Mdd k[\A]$ on the additive functors, and by 
$$\Ext^*_{k\otimes_{\mathbb{Z}}\A}(A,B)\;,\text{ and } \Tor_*^{k\otimes_{\mathbb{Z}}\A}(A,B)$$
the $\Ext$ and $\Tor$ in the additive setting. That is, they are the derived functors of the functors
$\Hom_{k[\A]}(A,-), A\otimes_{k[\A]}-: k\otimes_\mathbb{Z}\A\Md\to k\Md$.

\begin{rk}
Since additive functors are an abelian subcategory of all functors, we have canonical comparison maps:
$$\Ext^*_{k\otimes_\mathbb{Z}\A}(A,B)\to \Ext^*_{k[\A]}(A,B)\;, \text{ and } \Tor_*^{k[\A]}(A,B)\to \Tor_*^{k\otimes_\mathbb{Z}\A}(A,B)\;.$$
These comparison maps are usually far from being isomorphisms, but they are relatively well-understood if $k$ is a field, and if the $\Hom$ of $\A$ are vector spaces over the prime subfield of $k$, see \cite{DT-Add}.
\end{rk}
In the special case $k=\mathbb{Z}$, the categories of additive functors from $\A$ to $\mathbb{Z}$-modules will be denoted by $\A\Md$ and $\Mdd\A$. That is, we will simplify $\mathbb{Z}\otimes_\mathbb{Z}\A$ into $\A$ in the notations.

\subsection{Functors with several variables} 
We now review some classical results involving homological computations with functors of several variables. 
\begin{pr}\label{pr-no-Ext}
Let $B$ and $C$ be two objects of $k[\C\times\D]\Md$. Assume that one of the following two hypotheses holds.
\begin{enumerate}[(i)]
\item For all $x$ and $x'$ in $\C$ we have $\Ext^*_{k[\D]}(B(x,-),C(x',-))=0$.
\item For all $y$ and $y'$ in $\D$ we have $\Ext^*_{k[\C]}(B(-,y),C(-,y'))=0$.
\end{enumerate}
Then $\Ext^*_{k[\C\times\D]}(B,D)=0$.
\end{pr}
\begin{proof}
We only prove that hypothesis (ii) implies that $\Ext^*_{k[\C\times\D]}(B,D)=0$, the proof that hypothesis (i) implies the cancellation being similar.

We write $E_\C^*(y,y')=\Ext^*_{k[\C]}(B(-,y),C(-,y'))$ for short. 
There is an isomorphism, natural with respect to $B$, $C$, and the object $D$ of $k[\D^\op\times\D]\Md$:
\[\Hom_{k[\D^\op\times \D]}(D,E^0_\C)\simeq \Hom_{k[\C\times\D]}(B\otimes_{k[\D^\op]}D,C)\;.\qquad(*)\]
This isomorphism is the functor analogue of \cite[IX.2 Prop 2.2]{CE}, and we may construct it as follows. Firstly, there is a natural isomorphism when $D$ is a standard projective. Indeed $D(x,y)=k[\D(x,c)]\otimes k[\D(d,y)]$ and the two sides are naturally isomorphic to $E_\C^0(c,d)$ by the Yoneda lemma. Now the isomorphism extends to every functor $D$ by taking a projective presentation of $D$ by standard projectives.

Using the isomorphism $(*)$ we construct two spectral sequences converging to the same abutment (the construction of these spectral sequences is exactly the same as the one given for categories of modules in \cite[XVI.4]{CE}):
\begin{align*}
&\mathrm{I}^{p,q}=\Ext^p_{k[\D^\op\times\D]}(D,E^q_\C)\Rightarrow H^{p+q}\;,\\
&\mathrm{II}^{p,q}= \Ext^p_{k[\C\times\D]}(T_q^\D,C)\Rightarrow H^{p+q}\;,
\end{align*}
where $T_q^\D:\C\times\D\to k\Md$ is defined by
$T_q^\D(x,y)=\Tor_q^{k[\D^{\op}]}(B(x,-),D(-,y))$.
If $D=k[\D]$ then for all $y$, $k[\D](-,y)$ is a projective object of $k[\D^\op]\Md$, hence the functor $\Tor^{k[\D^\op]}_q(B,k[\D])$ is zero for positive $q$, and the dual Yoneda isomorphism  \eqref{eq-Yoneda-dual} shows that for $q=0$ this functor is isomorphic to $B$. Thus the second spectral sequence collapses at the second page and $H^*=\Ext^*_{k[\C\times\D]}(B,C)$. Hypothesis (ii) is nothing but $E^*_\C=0$, hence the first spectral sequence gives the result.
\end{proof}

If $\C$ is additive, we have a pair of functors 
$$\Delta:\C\leftrightarrows \C^{\times n}:\Sigma$$
defined by $\Delta(x)=(x,\dots,x)$ and $\Sigma(x_1,\dots,x_n)=x_1\oplus\cdots\oplus x_n$. We denote by $\delta:x\to \Sigma\Delta(x)=x^{\oplus n}$ the diagonal morphism, and by $\sigma: \Sigma\Delta(x)=x^{\oplus n}\to x$ the sum morphism (i.e. the components of $\delta$ and $\sigma$ are equal to the identity of $x$). 

\begin{pr}\label{pr-sum-diagonal} Assume that $\C$ is additive. 
For all $F$ in $k[\C]\Md$ and for all $G$ in $k[\C^{\times n}]\Md$ the following compositions are isomorphisms:
\begin{align*}
&\Ext^*_{k[\C^{\times n}]}(\Sigma^*F,G)\xrightarrow[]{\res^\Delta}\Ext^*_{k[\C]}(\Delta^*\Sigma^*F,\Delta^*G)\xrightarrow[]{\Ext^*(F(\delta),\Delta^*G)} \Ext^*_{k[\C]}(F,\Delta^*G)\;,\\
&\Ext^*_{k[\C^{\times n}]}(G,\Sigma^*F)\xrightarrow[]{\res^\Delta}\Ext^*_{k[\C]}(\Delta^*G,\Delta^*\Sigma^*F)\xrightarrow[]{\Ext^*(\Delta^*G,F(\sigma))} \Ext^*_{k[\C]}(F,\Delta^*G)\;.
\end{align*}
\end{pr}
\begin{proof}
The pair $(\Delta,\Sigma)$ is a pair of adjoint functors, with adjunction unit $\delta$ and the pair $(\Sigma,\Delta)$ is a pair of adjoint functors with adjunction counit $\sigma$. Thus, the result follows from \cite[Lm 1.5]{Pira-Pan}, with the explicit expression of the isomorphisms provided by \cite[Lm 1.3]{Pira-Pan}.
\end{proof}

Given functors $F$ in $k[\C]\Md$ and $G$ in $k[\D]\Md$, we denote by $F\boxtimes G$ the object of $k[\C\times\D]\Md$ defined by $(F\boxtimes G)(x,y)=F(x)\otimes G(y)$.
Tensor products yield a K\"unneth morphism:
\begin{align}\Ext^*_{k[\C]}(F,H)\otimes\Ext^*_{k[\D]}(G,K)\to \Ext^*_{k[\C\times\D]}(F\boxtimes G,H\boxtimes K)\;.\label{eq-Kunneth}\end{align}
A functor is called \emph{of type $\mathrm{fp}_\infty$} if it has a projective resolution by finite sums of standard projectives. 

The following property is well-known; we give its short proof for the convenience of the reader.
\begin{pr}\label{pr-Kunneth}
Assume that $k$ is a field and that $F$ and $G$ are of type $\mathrm{fp}_\infty$. Then the K\"unneth morphism \eqref{eq-Kunneth} is a graded isomorphism.
\end{pr}
\begin{proof}
If $F=P^x$ and $G=P^y$ are standard projectives, then $F\boxtimes G\simeq P^{(x,y)}$ is a standard projective, and it follows from the Yoneda isomorphism that the K\"unneth morphism is an isomorphism in this case. By additivity of $\Ext$ and tensor products, it follows that the K\"unneth morphism is an isomorphism if $F$ and $G$ are finite direct sums of standard projectives. Assume now that $P\to F$ and $Q\to G$ are two resolutions by finite direct sums of projectives. Then the K\"unneth morphism yields an isomorphism between the complex $\Hom_{k[\C]}(P,H)\otimes\Hom_{k[\D]}(Q,K)$ and the complex $\Hom_{k[\C\times\D]}(P\boxtimes Q,H\boxtimes K)$. Since $P\boxtimes Q$ is a projective resolution of $F\boxtimes G$, this implies the result.
\end{proof}
\begin{rk}\label{rk-pfinfty}
Being of type $\mathrm{fp}_\infty$ is a rather strong property, which is usually hard to check in an elementary way on a given functor. We refer the reader to \cite{DTSchwartz,DTTunisian} for practical conditions ensuring that functors are of type $\mathrm{fp}_\infty$.
\end{rk}

\section{Excision in functor homology}\label{sec-3}

The main purpose of this section is to prove theorem \ref{thm-intro-2} from the introduction. In fact, we will derive theorem \ref{thm-intro-2} from a general `excision theorem', which is a functor homology analogue of the excision theorem of Suslin and Wodzicki in $K$-theory \cite{Suslin-Wodz}, see remark \ref{rk-excis} for further explanations relative to this analogy. We finish the section by a quick computational application which is relevant for antipolynomial functors.
\subsection{The excision theorem}
\begin{defi-prop}\label{cor-memechose2}
Let $e$ be a positive integer or $+\infty$. A restriction functor $\phi^*:k[\D]\Md\to k[\C]\Md$ is called \emph{$e$-excisive} if it satisfies one of the following equivalent assertions.
\begin{enumerate}
\item[(1)] For all functors $F,G$, the map 
$$\res^\phi:\Ext^i_{k[\D]}(F,G)\to \Ext^i_{k[\C]}(\phi^*F,\phi^*G)$$
is an isomorphism if $0\le i< e$.
\item[(2)] For all functors $F',G$, the map 
$$\res_\phi:\Tor_i^{k[\C]}(\phi^*F',\phi^*G)\to \Tor_i^{k[\D]}(F',G)$$
 is an isomorphism if $0\le i< e$.
\item[(3)] The restriction functor $\phi^*:k[\D]\Md\to k[\C]\Md$ is fully faithful and for all objects $x$, $y$ of $\D$:
\[\bigoplus_{0<i< e}\Tor^{k[\C]}_i(\phi^*P^x_{\D^\op},\phi^*P^y_\D)=0\;.\] 
\end{enumerate}
\end{defi-prop}
\begin{proof}
Let us prove (1)$\Leftrightarrow$(2).
The map $\res^\phi$ is a morphism of $\delta$-functors, and every functor $G$ embeds into a product of standard injectives. Hence a d\'ecalage argument shows that assertion (1) is equivalent to the following assertion:
\begin{enumerate}
\item[(1')] For all standard injectives $G$ and for all functors $F$, $\res^\phi$ is an isomorphism in degrees $0\le *< e$.
\end{enumerate} 
Finally assertion (1') is equivalent to assertion (2) as a consequence of proposition \ref{pr-memechose} and the fact that a $k$-linear map $f$ is an isomorphism if and only if the $k$-linear map $\Hom_k(f,M)$ is an isomorphism for all injective $k$-modules $M$.

We now prove (2)$\Leftrightarrow$(3).
If (2) holds, then $\phi^*$ is fully faithful by (1), and moreover  
$\Tor^{k[\C]}_i(\phi^*P^x_{\D^\op},\phi^*P^y_{\D})$ is isomorphic to $\Tor^{k[\D]}_i(P^x_{\D^\op},P^y_{\D})$ for all $0<i< e$, hence it is zero by projectivity of $P^x_{\D^\op}$, which proves (3). Conversely, if $\phi^*$ is fully faithful then (1) holds for $e=1$, hence $\res_\phi$ is an isomorphism in degree $0$ by (2). If in addition the $\Tor$-vanishing is satisfied, then $\res_\phi$ is an isomorphism in degrees $0\le *< e$ for all projective functors $F'$ and $G$, hence for all functors $F'$ and $G$ by a d\'ecalage argument, which proves (2).
\end{proof}

\begin{ex}
If a functor $\phi:\C\to \D$ is full and essentially surjective, then the restriction functor  $\phi^*:k[\D]\Md\to k[\C]\Md$ is fully faithful, hence $1$-excisive.
\end{ex}

\begin{rk}
A standard $\delta$-functor argument shows that if $\phi^*$ is $e$-excisive for some positive integer $i$, then we automatically have in addition that in degree $i=e$, the map $\res^\phi$ is injective and the map $\res_\phi$ is surjective, for all functors $F'$, $F$ and $G$.
\end{rk}

The next theorem is the main result of the section, and we call it the excision theorem (see remark \ref{rk-excis} for an explanation of this terminology). It gives an equivalent condition to being $e$-excisive. 
The key point is that being $e$-excisive is a homological property of the category $k[\A]\Md$ of \emph{all} functors from $\A$ to $k$-modules, whereas the equivalent condition given in theorem \ref{thm-magique-general} deals with the homological properties of the category $\A\Md=\mathbb{Z}\otimes_\mathbb{Z}\A\Md$ of the \emph{additive} functors from $\A$ to abelian groups.
 
\begin{thm}[excision]\label{thm-magique-general}Let $k$ be a commutative ring and let $\phi:\A\to \B$ be an additive functor between two small additive categories, such that the restriction functor $\phi^*:k[\B]\Md\to k[\A]\Md$ is fully faithful.  
For all positive integers $e$, the following assertions are equivalent.
\begin{enumerate}[(1)]
\item\label{item-1-ex} The restriction functor $\phi^*:k[\B]\Md\to k[\A]\Md$ is $e$-excisive. 
\item\label{item-2-ex} For all objects $x$ and $y$ of $\B$, the torsion groups
$T_\bullet= \mathrm{Tor}_\bullet^{\A}(\phi^*h^x_{\B^\op},\phi^*h^y_{\B})$ (calculated in the category of additive functors from $\A$ to abelian groups)
satisfy 
\[
k \otimes_{\mathbb{Z}}T_i =0 =\mathrm{Tor}_1^{\mathbb{Z}}(k,T_{j})=0\;,
\text{ for $0<i<e$ and $0<j<e-1$.}\]
\end{enumerate}
\end{thm}

The remainder of the section is devoted to the proof of theorem \ref{thm-magique-general}. This proof depends on the use of simplicial techniques, and in particular on variants of the Hurewicz theorem. For the convenience of the reader, the simplicial notions and results that we need are recalled in the appendix (section \ref{sec-app}). The first step of the proof is the following general lemma, which is of independent interest. It is well-known to experts, but we do not know any written reference for it.

\begin{lm}\label{lm-kcrochet-tens}
Let $A:\A^\mathrm{op}\to \mathbb{Z}\Md$ and $B:\A\to \mathbb{Z}\Md$ be two additive functors, and let $k[A]$ and $k[B]$ denote the composition of these functors with the $k$-linearization functor $k[-]$. There is an isomorphism of $k$-modules, natural with respect to $A$ and $B$:
$$k[A]\otimes_{k[\A]} k[B]\simeq k[A\otimes_{\mathbb{Z}[\A]} B]\;.$$
\end{lm}
\begin{proof}
We first recall concrete formulas for tensor products.
For all commutative rings $K$, the tensor product $F\otimes_{K[\A]}G$ can be concretely computed as the quotient of the direct sum $\bigoplus_{x} F(x)\otimes_K G(x)$ indexed by a set of representatives of the isomorphism classes of objects of $\A$, modulo the relations $F(f)(s)\otimes t = s\otimes G(f)(t)$ for all morphisms $f:x\to y$ and for all elements $s\in F(y)$ and $t\in G(x)$. We will denote by $\llbracket s\otimes t\rrbracket\in F\otimes_{K[\A]}G$ the class of an element $s\otimes t\in F(x)\otimes_K G(x)$. When $G=P^c$ is a standard projective there is a `Yoneda isomorphism' 
$$\Upsilon:F\otimes_{K[\A]}P^c\simeq F(c)$$ given by sending the class $\llbracket s\otimes f\rrbracket $ with $s\in F(x)$ and $f\in\A(c,x)$ to $F(f)(s)$ (the inverse isomorphism sends $u\in F(c)$ to $\llbracket u\otimes \id_c\rrbracket$). Similarly, if $F$ is additive and $G=h^c$ is a standard additive projective, there is an `additive Yoneda isomorphism' 
$$\Upsilon_\mathrm{add}:F\otimes_{K[\A]}h^c\simeq F(c)\;.$$ 
When $k=\mathbb{Z}$, the isomorphism $\Upsilon_{\mathrm{add}}$ sends class $\llbracket s\otimes f\rrbracket $ with $s\in F(x)$ and $f\in\A(c,x)$ to $F(f)(s)$. 

We are now ready to construct the isomorphism of lemma \ref{lm-kcrochet-tens}. 
For all objects $x$ of $\A$, we let $\theta_{A,B,x}:k[A(x)]\otimes k[B(x)]\to k[A\otimes_{\mathbb{Z}[\A]}B]$ be the $k$-linear map
such that $\theta_{A,B,x}(s\otimes t)=\llbracket s\otimes t\rrbracket$ for all $s$ in $A(x)$ and all $t\in B(x)$. The maps $\theta_{A,B,x}$ induce a $k$-linear map, natural in $A$ and $B$:
\[\Theta_{A,B}:k[A]\otimes_{k[\A]}k[B]\to k[A\otimes_{\mathbb{Z}[\A]}B]\;.\]

If $B=h^c$ is a standard projective additive, the composition $k[\Upsilon_{\mathrm{add}}]\circ \Theta_{A,B}$ is equal to $\Upsilon$, hence $\Theta_{A,B}$ is an isomorphism in this case.
Finite direct sums of standard projective additives are isomorphic to standard projective additives, hence $\Theta_{A,B}$ is also an isomorphism if $B$ is a finite direct sum of standard additive projectives. Now the source and the target of $\Theta_{A,B}$, viewed as functors of $B$ preserve filtered colimits of monomorphisms, which implies in turn that $\Theta_{A,B}$ is an isomorphism if $B$ is an arbitrary direct sum of standard projective additives.

Now let $B$ be arbitrary and let $\epsilon:P\to B$ be a projective simplicial resolution of $B$ in $\A\Md$ by direct sums of standard projective additives. Then we have a commutative square of simplicial $k$-modules in which the top row is an isomorphism, the bottom row features constant simplicial $k$-modules and the vertical arrows are induced by $\epsilon$:
\[
\begin{tikzcd}
k[A]\otimes_{k[\A]}k[P]\ar{d}\ar{r}{\Theta_{A,P}}[swap]{\simeq}& k[A\otimes_{\mathbb{Z}[\A]}P]\ar{d}\\
k[A]\otimes_{k[\A]}k[B]\ar{r}{\Theta_{A,B}}& k[A\otimes_{\mathbb{Z}[\A]}B]
\end{tikzcd}
\;.\]
By right exactness of tensor products and by the classical Hurewicz theorem (see proposition \ref{pr-Hur-classical}) the vertical morphisms are isomorphisms in $\pi_0$. Hence $\Theta_{A,B}$ is an isomorphism.
\end{proof}

\begin{proof}[Proof of theorem \ref{thm-magique-general}]
Let $Q$ be a simplicial resolution by direct sums of standard projectives in $\Mdd\A$ of the additive functor $\phi^*h^x_{\B^\op}=\B(\phi(-),x)$. By definition of $\Tor^\A_\bullet$, the abelian group $T_i$ is the $i$-th homotopy group of the simplicial abelian group $X:=Q\otimes_{\mathbb{Z}[\A]}\phi^*h^y_\B$:
$$T_\bullet=\pi_\bullet X\;.$$

We are now going to give an interpretation of the homotopy groups of $k[X]$.  
Lemma \ref{lm-kcrochet-tens} yields an isomorphism of simplicial $k$-modules:
\[k[X]\simeq k[Q]\otimes_{k[\A]}k[\phi^*h_\B^y]\;. \]
We first observe that $k[\phi^*h_\B^y]=\phi^*P^y_\B$. Next, we claim that $k[Q]$ is a flat resolution of $k[\phi^*h^x_{\B^\op}]=\phi^*P^x_{\B^\op}$. Indeed, it follows from the Hurewicz theorem \ref{pr-Hur-classical} that $k[Q]$ is a simplicial resolution of $k[\phi^*h^x_{\B^\op}]=\phi^*P^x_{\B^\op}$. Moreover, if $A=h^c$ is a standard projective in $\Mdd\A$, then $k[A]=P^c$ is a standard projective in $\Mdd k[\A]$. More generally if $A$ is an arbitrary direct sum of standard projectives in $\Mdd\A$, then $A$ is the filtered colimit of its standard projective subfunctors, and since the $k$-linearization functor $k[-]$ preserves filtered colimits of monomorphisms, we obtain that $k[A]$ is a filtered limit of projectives, hence that $k[A]$ is flat. Therefore the simplicial resolution $k[Q]$ is degreewise flat, hence we have a graded isomorphism:
\[ \Tor^{k[\A]}_\bullet(\phi^*P^x_{\B^\op},\phi^*P^y_\B)= \pi_\bullet k[X]\;.\]

To finish the proof, we observe that the $k$-local Hurewicz theorem (see corollary \ref{cor-Hur}) shows that assertion (2) of the theorem is equivalent to the vanishing of $\Tor^{k[\A]}_i(\phi^*P^x_{\B^\op},\phi^*P^y_\B)$ for $0<i<e$ and for all objects $x$ and $y$ of $\B$, and the latter is equivalent to assertion (1) by proposition \ref{cor-memechose2}.
\end{proof}

\subsection{Two special cases of the excision theorem}
We now investigate some concrete situations in which condition \eqref{item-2-ex} of theorem \ref{thm-magique-general} is satisfied.
As a first example, condition \eqref{item-2-ex} is always satisfied if the category $\B$ is such that $\B(x,x)\otimes_{\mathbb{Z}} k=0$ for all $x$, hence in this case theorem \ref{thm-magique-general} takes the following form. 

\begin{thm}\label{thm-magique-csq1}
Let $\phi:\A\to \B$ be an additive quotient. Assume that for all $x$, $\B(x,x)\otimes_\mathbb{Z}k=0$. Then the restriction functor $\phi^*:k[\B]\Md\to k[\A]\Md$
is $\infty$-excisive.
\end{thm}

\begin{proof}
Let $\C$ denote the following full subcategory of abelian groups:
\begin{itemize}
\item if $\operatorname{char} k\ne 0$, the objects of $\C$ are the groups on which multiplication by $\operatorname{char} k$ is invertible;
\item if $\operatorname{char} k=0$, the objects of $\C$ are the torsion groups whose elements have orders invertible in $k$.
\end{itemize}
This subcategory is stable under kernels, cokernels and direct sums. Moreover, for all $A\in \C$ we have $\Tor^\mathbb{Z}_1(A,k)=0=A\otimes_\mathbb{Z} k$.

The hypothesis on $\B$ implies that $\B(x,y)\in\C$ for all $(x,y)$, thanks to lemma~\ref{exoann} below (applied to the rings $k$ and $\B(y,y)$, using that $\B(x,y)$ is a $\B(y,y)$-module). Thus, for all standard projectives $h^a_{\A^\op}$ in $\Mdd\A$, the abelian group $h^a_{\A^\op}\otimes_\A \phi^*h^y_\B=\B(y,\phi(a))$ belongs to $\C$. Therefore, if $Q$ is a resolution of $\phi^*h^x_{\B^\op}$ in $\Mdd\A$ by direct sums of standard projectives, the complex $Q\otimes_{k[\A]}\phi^*h^y_\B$ calculating $\Tor^{\A}_\bullet(\phi^*h^x_{\B^\op}, \phi^*h^y_\B)$ is a complex in $\C$, hence its homology groups are in $\C$. Thus the second assertion of theorem \ref{thm-magique-general} is satisfied for all $e$, whence the result.
\end{proof}

\begin{lm}\label{exoann} Let $R$ and $S$ be rings such that $R\otimes_\mathbb{Z}S=0$. Let us denote $r:=\operatorname{char} R$ and $s:=\operatorname{char} S$. Then $(r,s)\ne (0,0)$. Moreover, if $r\ne 0$, then $r$ belongs to $S^\times$.
\end{lm}

\begin{proof} If a tensor product of abelian groups is zero, at least one of them is torsion, whence $(r,s)\ne (0,0)$. If $r\ne 0$, then $\mathbb{Z}/r$ is a direct summand of the additive group of $R$, whence $\mathbb{Z}/r\otimes_\mathbb{Z}S=0$, which implies $r\in S^\times$.
\end{proof}

Under some favorable assumptions on $\A$ and $\B$, assertion \eqref{item-2-ex} of theorem \ref{thm-magique-general} can also be reformulated in terms of excision for \textit{additive} functors. We first transpose definition \ref{cor-memechose2} to the context of additive functors (the proof is the same as the proof of proposition \ref{cor-memechose2} so we don't repeat it).  
\begin{defi-prop}
Let $e$ be a positive integer or $+\infty$. Let $\phi:\A\to \B$ be an additive functor between two essentially small additive categories. The restriction functor $\phi^*:k\otimes_{\mathbb{Z}}\B\Md\to k\otimes_{\mathbb{Z}}\A\Md$ is called \emph{$e$-excisive} if it satisfies one of the following equivalent assertions.
\begin{enumerate}
\item[(1)] For all functors $F,G$, the map 
$$\res^\phi:\Ext^i_{k\otimes_{\mathbb{Z}}\B}(F,G)\to \Ext^i_{k\otimes_{\mathbb{Z}}\A}(\phi^*F,\phi^*G)$$
is an isomorphism if $0\le i< e$.
\item[(2)] For all functors $F',G$, the map 
$$\res_\phi:\Tor_i^{k\otimes_{\mathbb{Z}}\A}(\phi^*F',\phi^*G)\to \Tor_i^{k\otimes_{\mathbb{Z}}\B}(F',G)$$
 is an isomorphism if $0\le i< e$.
\item[(3)] The restriction functor $\phi^*:k\otimes_{\mathbb{Z}}\B\Md\to k\otimes_{\mathbb{Z}}\A\Md$ is fully faithful and for all objects $x$, $y$ of $\B$:
\[\bigoplus_{0<i< e}\Tor^{k\otimes_{\mathbb{Z}}\A}_i(\phi^*h^x_{\B^\op},\phi^*h^y_\B)=0\;.\] 
\end{enumerate}
\end{defi-prop}

\begin{defi}
Let $k$ be a commutative ring. We say that an additive category $\C$ is \emph{$k$-torsion-free} if $\mathrm{Tor}^{\mathbb{Z}}_1(k,\C(x,y))=0$ for all objects $x$, $y$ of $\C$. 
\end{defi}

\begin{thm}\label{thm-magique2} Let $\phi:\A\to \B$ be an additive quotient.  Assume that $\A$ and $\B$ are both $k$-torsion free. Then for all positive integers $e$, the following assertions are equivalent.
\begin{enumerate}[(1)]
\item The functor $\phi^*:k[\B]\Md\to k[\A]\Md$ is $e$-excisive.
\item The functor $\phi^*:k\otimes_{\mathbb{Z}}\B\Md\to k\otimes_{\mathbb{Z}}\A\Md$ is $e$-excisive.
\end{enumerate}
\end{thm}
\begin{proof}
Let $Q^x$ be a resolution of the functor $\B(\phi(-),x)$ in $\Mdd\A$ by direct sums of standard projectives. Then the homology of the complex 
$$C^{x,y}:=Q^x\otimes_{\mathbb{Z}[\A]}\B(y,\phi(-))$$
computes $\Tor_\bullet^{\A}(\B(\phi(-),x),\B(y,\phi(-)))$. 
We claim that the complex $k\otimes_\mathbb{Z}Q^x$ is a resolution of $k\otimes_\mathbb{Z}\B(\phi(-),x)$ in $\Mdd k\otimes_{\mathbb{Z}}\A$. Indeed, $\Tor^\mathbb{Z}_1(k,-)$ vanishes on the objects of $Q$ since $\A$ is $k$-torsion-free, and also on $\pi_0Q$ because $\B$ is $k$-torsion-free. Thus the claim follows from the universal coefficient theorem \cite[XII Thm 12.1]{MLHom}. Moreover, $k\otimes_\mathbb{Z}Q^x$ is a direct sum of standard projectives in each degree. Therefore, the complex
$$k\otimes_{\mathbb{Z}}C^{x,y} \simeq (k\otimes_{\mathbb{Z}}Q^x)\otimes_{k[\A]}(k\otimes_{\mathbb{Z}}\B(y,\phi(-))$$
computes $\Tor_\bullet^{k\otimes_{\mathbb{Z}}\A}(k\otimes_{\mathbb{Z}}\B(\phi(-),x),k\otimes_{\mathbb{Z}}\B(y,\phi(-))$. 

Now, $\Tor^\mathbb{Z}_1(k,-)$ vanishes on the tensor product $h^a\otimes_{k[\A]}A\simeq \B(y,\phi(a))$
for all standard projectives $h^a=\A(-,a)$ of $\Mdd \A$ (once again because $\B$ is $k$-torsion-free), hence on the objects of the complex $C^{x,y}$. As a consequence, the universal coefficient theorem \cite[XII Thm 12.1]{MLHom} yields 
short exact sequences
$$0\to k\otimes_\mathbb{Z}H_i(C^{x,y})\to H_i(k\otimes_\mathbb{Z}C^{x,y})\to \mathrm{Tor}^{\mathbb{Z}}_1(k,H_{i-1}(C^{x,y}))\to 0$$
for all integers $i$. 
Therefore assertion (2) is equivalent to the vanishing of $H_i(k\otimes_\mathbb{Z}C^{x,y})$ for $0<i<e$ and all $x$ and $y$, which is equivalent to the vanishing of $k\otimes_\mathbb{Z}H_i(C^{x,y})$ and $\mathrm{Tor}^{\mathbb{Z}}_1(k,H_{i-1}(C^{x,y}))$ for $0<i<e$ and all $x$ and $y$. 

Since $\phi$ is full and essentially surjective, $\phi^*: \mathbb{Z}[\B]\Md\to \mathbb{Z}[A]\Md$ is $1$-excisive, hence 
$$H_0(C^{x,y})=\B(\phi(-),x)\otimes_{\mathbb{Z}[\A]}\B(y,\phi(-))\simeq \B(-,x)\otimes_{\mathbb{Z}[\B]}\B(y,-)\simeq \B(y,x)$$
and since $\B$ is assumed to be $k$-torsion-free we always have $\mathrm{Tor}^{\mathbb{Z}}_1(k,H_{i-1}(C^{x,y}))=0$. As a consequence, assertion (2) is equivalent to the vanishing of $k\otimes_\mathbb{Z}H_i(C^{x,y})$ and $\mathrm{Tor}^{\mathbb{Z}}_1(k,H_{j}(C^{x,y}))$ for $0<i<e$ and $0<j<e-1$. The latter is equivalent to assertion (1) by theorem \ref{thm-magique-general}.
\end{proof}

\begin{rk}\label{rk-excis}
Theorem \ref{thm-magique2} is a functor homology analogue of Suslin-Wodzicki's excision theorem in rational algebraic $K$-theory \cite{Suslin-Wodz} (see also \cite{Suslin-Kexcision} for the non-rational case). Indeed, the second assertion in theorem \ref{thm-magique2} is a natural generalization of the `$H$-unital' condition which governs excision in $K$-theory.

To be more specific, if $I$ is a two-sided ideal of a ring $R$, and if we consider $\A=\Proj_R$, $\B=\Proj_{R/I}$ and $\phi=-\otimes_R R/I$, then the second assertion of theorem \ref{thm-magique2} is easily seen to be equivalent to 
$$\text{(3)}\quad 
\mathrm{Tor}^{R\otimes_{\mathbb{Z}} k}_i((R/I)\otimes_{\mathbb{Z}} k,(R/I)\otimes_\mathbb{Z} k)=0 \text{ for $0<i<e$.}
$$ 
(To prove the equivalence, use proposition \ref{cor-memechose2} and the fact that $k\otimes_{\mathbb{Z}}\A\Md$ is equivalent to $R\otimes_\mathbb{Z}k\Md$ by the Eilenberg-Watts theorem.)
In the situation considered in \cite{Suslin-Wodz}, that is, if  
$R=\mathbb{Z}\oplus I$ where $I$ is a ring without unit (seen as an ideal in the unital ring $R$ constructed by adding formally a unit to $I$) and $k=\mathbb{Q}$, the $\Tor$ appearing in assertion (3) can be computed with a bar complex, hence assertion (3) is equivalent to $R$ being $H$-unital. 
\end{rk}
\subsection{An application to a vanishing result}

\begin{pr}\label{pr-vanish}
Let $k$ be a field of characteristic zero, and let $F,F',G$ be three functors from $\A$ to $k\Md$, with $F$ contravariant. Assume that there is a finite semi-simple ring $R$ such that $\Proj_R$ is an additive quotient of $\A$, and such that $F,F',G$ factor through $\Proj_R$. Then for all $i>0$ we have:
\begin{align*}
&\Ext^i_{k[\A]}(F',G)=0\;, &&  \Tor_i^{k[\A]}(F,G)=0\;.
\end{align*}
\end{pr}
\begin{ex}
If $\A$ is the category of finitely generated abelian groups, then the categories of finite dimensional vector spaces over a finite field $\Fp$ of prime cardinal $p$ are additive quotients of $\A$. The proposition implies for example that the functor $F:\A\to \mathbb{Q}\Md$ defined by $F(A)=\mathbb{Q}[A/pA]$ has no self-extensions of positive degree.
\end{ex}

The proof of proposition \ref{pr-vanish} is a direct application of theorem \ref{thm-magique-csq1} together with the following consequence of Kuhn's structure results \cite{Ku-adv}.
\begin{pr}\label{pr-vanish-Kuhn}
If $R$ is a finite semi-simple ring and if $k$ is field of characteristic zero, then every object of $k[\Proj_R]\Md$ is projective and semi-simple. 
\end{pr} 
\begin{proof}
Since $R$ is a finite semi-simple ring, then $R$ is isomorphic to a product $R_1\times\cdots\times R_n$ of finite simple rings, thus $\Proj_R$ is equivalent to  $\Proj_{R_1}\times\cdots\times \Proj_{R_n}$. Moreover each finite simple ring $R_i$ is isomorphic to a matrix ring over a certain finite field $\mathbb{F}(i)$. Hence by Morita theory, $\Proj_R$ is equivalent to $\Proj_{\mathbb{F}(1)}\times\cdots\times \Proj_{\mathbb{F}(n)}$. Thus we may replace $k[\Proj_R]\Md$ by $k[\Proj_{\mathbb{F}(1)}\times\cdots\times \Proj_{\mathbb{F}(n)}]\Md$ in this proof.
Every standard projective $P$ in the latter category can be written as:
$$P(x_1,\dots,x_n)\simeq P_1(x_1)\otimes \cdots\otimes P_n(x_n)\qquad(*)$$ 
where each $P_i$ is a standard projective of $k[\Proj_{\mathbb{F}(i)}]\Md$. 
The main result of \cite{Ku-adv} says that  $k[\Proj_{\mathbb{F}(i)}]\Md$ is equivalent to the infinite product $\prod_{n\ge 0}k[\GL_n(\mathbb{F}(i))]\Md$ and Maschke's theorem implies that $P_i$ a direct sum of simple functors. Therefore decomposition $(*)$ implies that $P$ is a direct sum of functors of the form $S_1(x_1)\otimes \cdots \otimes S_n(x_n)$ where each $S_i$ is simple, and by \cite[Cor 3.13 and Prop 3.7]{DTV} such functors are simple. We conclude that every standard projective of $k[\Proj_{\mathbb{F}(1)}\times\cdots\times \Proj_{\mathbb{F}(n)}]\Md$ is semi-simple, hence that every object of $k[\Proj_{\mathbb{F}(1)}\times\cdots\times \Proj_{\mathbb{F}(n)}]\Md$ is projective and semi-simple. 
\end{proof}

\section{A polynomial version of excision}\label{sec-4}

Let us start by recalling the classical concepts of reduced functors and polynomial functors defined on an essentially small additive category $\A$. Much of the material explained in this section will also be used for the separation theorem proved in section \ref{sec-sep}.

\subsection{Reduced functors}
A functor $F:\A\to k\Md$ is \emph{reduced} if it satisfies $F(0)=0$. We define the reduced part $F^\red$ of a functor $F$ by: 
$$F^\red(x):=\mathrm{Ker} (F(0):F(x)\to F(0))\;.$$
We have a decomposition $F\simeq F^\red\oplus F(0)$, where $F(0)$ denotes the constant functor with value $F(0)$. This decomposition is natural with respect to $F$ and induces
natural isomorphisms:
$$\Ext^*_{k[\A]}(F,G)\simeq\Ext^*_k(F(0),G(0))\oplus \Ext^*_{k[\A]}(F^\red,G^\red)\;.$$
In particular, there is no nonzero $\Ext$ between a constant functor and a reduced functor. This $\Ext$-vanishing between constant functors and reduced functors can be generalized to functors with several variables by applying proposition \ref{pr-no-Ext}. Namely we obtain the following vanishing result.
\begin{lm}\label{lm-vanish-red}
Let $\A_1, \dots, \A_n$ be essentially small additive categories, and let $F$ and $G$ be objects of $k[\A_1\times\cdots\A_n]\Md$. Assume that there is an $i$ such that $F$ is constant with respect to its $i$-th variable and such that $G$ vanishes if its $i$-th variable is zero. Then 
$\Ext^*_{k[\A_1\times\cdots\times\A_n]}(F,G)=0= \Ext^*_{k[\A_1\times\cdots\times\A_n]}(G,F)$.
\end{lm}

\subsection{Polynomial functors}\label{subsec-pol}

The notion of \textit{cross-effects} was introduced by Eilenberg and MacLane. To be more specific, they proved \cite[Thm 9.6]{EML} that for all functors $F:\A\to k\Md$, and for all positive integers $d$ the functor on $d$ variables $F(x_1\oplus\cdots \oplus x_d)$ has (up to isomorphism) a unique decomposition:
\begin{equation}F(x_1\oplus\cdots\oplus x_d)\simeq F(0)\oplus \bigoplus_{\sigma}F_\sigma(x_{\sigma_1},\dots,x_{\sigma_r})\label{eq-decomp-pol}
\end{equation}
where the sum runs over the non-void subsets $\sigma=\{\sigma_1,\dots,\sigma_r\}$ of $\{1,\dots,d\}$, and where the $F_\sigma$ are functors of $r$ variables which are zero whenever one of their variables is zero. The functor of $d$ variables $F_{\{1,\dots,d\}}:\A^{\times d}\to k\Md$ is called the $d$-th cross-effect of $F$, and we shall denote it by $\Cr_dF$. 
The functor $F$ is \emph{polynomial of degree less than $d$} if its $d$-th cross-effect is zero. This is equivalent to the fact that all the functors $\Cr_kF$ for $k\ge d$ are zero.

If $M$ is a $k$-module, applying $\Hom_k(-,M)$ to the decomposition \eqref{eq-decomp-pol} yields a similar decomposition for the dual functor $D_MF$. This proves that $\Cr_d(D_MF)\simeq D_M(\Cr_dF)$. As a consequence, $F$ is polynomial of degree less than $d$ if and only if $D_MF$ is polynomial of degree less than $d$ for all injective $k$-modules $M$.

The next proposition explains why the concept of polynomial functors is relevant for homological computations. The implication (i)$\Rightarrow$(iii) is known as Pirashvili's vanishing lemma. This vanishing lemma first appeared in \cite{Pira-Higher} and it has been widely used in homological computations. 
\begin{pr}\label{pr-eq}
Let $F:\A\to k\Md$ be a functor and let $d$ be a positive integer. The following four assertions are equivalent.
\begin{enumerate}[(i)]
\item $F$ is polynomial of degree less than $d$,
\item $\Hom_{k[\A]}(F_1\otimes\cdots\otimes F_d,F)=0$ for all reduced functors $F_1$,\dots,$F_d$, 

\item $\Ext^*_{k[\A]}(F_1\otimes\cdots\otimes F_d,F)=0$ for all reduced functors $F_1$,\dots,$F_d$,
\item $\Tor_*^{k[\A]}(F_1\otimes\cdots\otimes F_d,F)=0$ for all reduced functors $F_1$,\dots,$F_d$.
\end{enumerate}
\end{pr}

\begin{proof}
The equivalences (i)$\Leftrightarrow$ (ii)$\Leftrightarrow$ (iii) are already proved in \cite[Cor 2.17]{DTV} -- to be more specific, \cite[Cor 2.17]{DTV} assumes that $k$ is the field, but the proof carries out verbatim over an arbitrary commutative ring $k$.
So we bound ourselves to proving that (iv) is equivalent to the first three assertions. By proposition \ref{pr-memechose} (iv) is equivalent to 
$\Ext^*_{k[\A^\op]}(F_1\otimes\cdots\otimes F_d,D_MF)$ being zero
for all injective $k$-modules $M$, which is equivalent to $D_MF$ being polynomial of degree less than $d$ for all injective $k$-modules $M$, which is equivalent to $F$ being polynomial of degree less than $d$.
\end{proof}

\subsection{Polynomial excision} The next statement is an analogue of theorem \ref{thm-magique-csq1} which applies to polynomial functors. Observe that the hypothesis on $\phi$ in the polynomial excision theorem is essentially `orthogonal' to the hypothesis on $\phi$ in theorem \ref{thm-magique-csq1}. A typical situation in which theorem \ref{thm-poly-excis} applies is when $n=0$ in $k$ and $\B$ is the additive quotient of $\A$ obtained by modding out each abelian group $\A(x,y)$ by its subgroup $n\A(x,y)$.  
\begin{thm}[Polynomial excision]\label{thm-poly-excis}
Let $\phi:\A\to \B$ be an additive quotient. Assume that for all $x$, $\phi$ induces isomorphisms 
$$\A(x,x)\otimes_\mathbb{Z}k\simeq \B(\phi x,\phi x)\otimes_\mathbb{Z}k \quad\text{ and }\quad \Tor_1^\mathbb{Z}(\A(x,x),k)\simeq \Tor_1^\mathbb{Z}(\B(\phi x,\phi x),k)\;.$$
Then for all functors $G$ and for all polynomial functors $F$, restriction along $\phi$ induces graded isomorphisms
\begin{align*}
&\res^\phi\,:\; \Ext^\bullet_{k[\B]}(G,F)\simeq \Ext^\bullet_{k[\A]}(\phi^*G,\phi^*F)\;,\\
&\res_\phi\,:\; \Tor_\bullet^{k[\A]}(\phi^*G,\phi^*F)\simeq \Tor_\bullet^{k[\B]}(G,F) \;.
\end{align*}
\end{thm}

The proof of theorem \ref{thm-poly-excis} depends on the following vanishing lemma, which will be also useful in the proof of the separation theorem in section \ref{sec-sep}.

\begin{lm}\label{lm-vanish-pol}
Let $G:\A\to \mathbb{Z}\Md$ be a functor such that $\Tor^\mathbb{Z}_1(k,G(x))=0$ and $k\otimes_{\mathbb{Z}}G(x)=0$ for all objects $x$ of $\A$, and let $k[G]$ be its composition with the $k$-linearization functor $k[-]$. For all functors $H$ and for all polynomial functors $F$ we have:
$$\Ext^*_{k[\A]}(k[G]^\red\otimes H,F)=0\;.$$
\end{lm}
\begin{proof}
The functor $k[G]$ is isomorphic to $k[G(0)]\otimes k[G^\red]$, hence to a direct sum of copies of $k[G^\red]$. Therefore it suffices to prove the lemma when $G$ is reduced.

If $G$ is reduced, the reduced functor $k[G]^\red$ is equal to the functor $I$, which sends every $x$ to the augmentation ideal $I(x)$ of the $k$-algebra $k[G(x)]$ of the abelian group $G(x)$.   
By lemma \ref{lm-EML-vanish}, the abelian group $G(x)$ has trivial homology with coefficients in $k$. Thus, the reduced normalized bar construction of $k[G(x)]$ yields an exact complex 
$$ \dots\to I^{\otimes i+1}\to I^{\otimes i}\to \dots \to I^{\otimes 2}\to I\to 0\;.$$ 
Tensoring this complex by $I^{\otimes r-1}\otimes H$ yields a (non projective) resolution of $I^{\otimes r}\otimes H$ with associated hypercohomology spectral sequence:
$$E_1^{s,t}(r)=\Ext^t_{k[\A]}(I^{\otimes s+r+1}\otimes H,F)\Rightarrow \Ext^{s+t}_{k[\A]}(I^{\otimes r}\otimes H,F)\;.$$
We use this spectral sequence to that $\Ext^*_{k[\A]}(I^{\otimes r}\otimes H,F)$ vanishes for all positive $r$, by downward induction on $r$. Since $I^{\otimes r}\otimes H$ is the direct sum of $I^{\otimes r}\otimes H^\red$ and $I^{\otimes r}\otimes H(0)$, the vanishing holds for $r>\deg F$ by proposition \ref{pr-eq}. Now if the vanishing holds for a given $r$, then $E^1_{*,*}(r-1)=0$, hence $\Ext^*_{k[\A]}(I^{\otimes r-1}\otimes H,F)=0$. The result follows.
\end{proof}

\begin{proof}[Proof of theorem \ref{thm-poly-excis}]
We first prove the $\Ext$-isomorphism. Since $\phi$ is essentially surjective and fully faithful, the restriction functor $\phi^*:k[\B]\Md\to k[\A]\Md$ is fully faithful. Hence, in order to prove that $\res^\phi$ is an isomorphism, we only have to prove that  
$$\Ext^i_{k[\A]}(\phi^*P,\phi^*F)=0\qquad(*)$$
for all positive $i$ and for all projective objects $P$ of $k[\B]\Md$. But every projective object of $k[\B]\Md$ is a direct summand of a direct sum of standard projectives, and $\phi$ is essentially surjective, hence it suffices to prove $(*)$ when $P=k[\B(\phi(x),-)]$ with $x$ and object of $\A$. 

Let $\J(x,-)$ be the subfunctor of $\A(x,-):\A\to \mathbb{Z}\Md$ such that
$$\J(x,y)=\{f\in\A(x,y)\;,\; \phi(f)=0\}\;.$$ 
Then the functor $\phi^*P$ is isomorphic to the functor $k[\A(x,-)/\J(x,-)]$. Moreover, the hypotheses of the theorem imply that $\phi$ induces isomorphisms
$$\A(x,y)\otimes_\mathbb{Z}k\simeq \B(\phi x,\phi y)\otimes_\mathbb{Z}k \quad\text{ and }\quad \Tor_1^\mathbb{Z}(\A(x,y),k)\simeq \Tor_1^\mathbb{Z}(\B(\phi x,\phi y),k)$$
for all $y$, hence that for all $y$
$$\J(x,y)\otimes_\mathbb{Z}k = 0=\Tor_1^\mathbb{Z}(\J(x,y),k)\;.$$

Now if $A$ is an abelian group and $J$ is a subgroup of $A$, then $k[A]$ is a free $k[J]$-module, and $k\otimes_{k[J]}k[A]\simeq k[A/J]$. Therefore the normalized bar construction yields an exact complex, where $I\subset k[J]$ denotes the augmentation ideal of $k[J]$:
$$\dots\to \underbrace{I^{\otimes i}\otimes k[A]}_{\deg i}\to I^{\otimes i-1}\otimes k[A] \to \dots\to \underbrace{k[A]}_{\deg 0}\to \underbrace{k[A/J]}_{\deg -1}\to 0\;.$$
We consider this complex with $J=\J(x,-)$ and $A=\A(x,-)$. The resulting complex is a complex of functors with $\phi^*P$ in degree $-1$, which is a standard projective in degree $0$, and whose terms of positive degrees have no nontrivial $\Ext$ against $\phi^*F$ by lemma \ref{lm-vanish-pol} (and because $\phi^*F$ is polynomial, as the composition of a polynomial functor with an additive functor).
Therefore, the vanishing property $(*)$ holds by a standard $\delta$-functor argument.

This proves that $\res^\phi$ is a graded isomorphism. The fact that $\res_\phi$ is a graded isomorphism now follows by duality (use proposition \ref{pr-memechose} and the fact that $D_MF$ is polynomial if $F$ is polynomial).
\end{proof}

\begin{ex} Let $p$ be a prime, $V$ be any $\mathbb{Z}[1/p]$-module and $R:=\mathbb{Z}\ltimes V$ denote the square-zero extension of the ring of integers by $V$. The quotient map $R\twoheadrightarrow\mathbb{Z}$ induces an additive quotient functor $\phi: \Proj_R\to\Proj_\mathbb{Z}$. If $k$ is a ring of characteristic $p$, then $\phi$ fulfils the assumptions of Theorem~\ref{thm-poly-excis}, thus $\res^\phi$ and $\res_\phi$ are isomorphisms.
\end{ex}

%

\section{Separation}\label{sec-sep}

\subsection{The separation theorem}
The notion of a polynomial functor is recalled in section \ref{subsec-pol}. We now recall the definition of an antipolynomial functor. An \emph{ideal} of the additive category $\A$ is \cite[p.18]{Mi72} a subfunctor of $\A(-,-):\A^\op\times \A\to \mathbf{Ab}$. Given such an ideal $\I$, we can form the additive quotient $\A/\I$ of $\A$, with the same objects as $\A$ and with morphisms $({\A/\I})(x,y)=\A(x,y)/\I(x,y)$. We let $\pi_\I:\A\to \A/\I$ denote the additive quotient functor. The next definitions were first introduced in {\cite[section 4.1]{DTV}}.
\begin{defi}
An additive category $\B$ is \emph{$k$-trivial} if for all objects $x$ and $y$ the abelian group $\B(x,y)$ is finite and such that $k\otimes_\mathbb{Z}\B(x,y)=0$.  
An ideal $\I$ of $\A$ is \emph{$k$-cotrivial} if $\A/\I$ is $k$-trivial.
\end{defi}
\begin{defi}\label{defi-A}
A functor $F:\A\to k\Md$ is \emph{antipolynomial}  if there is a $k$-cotrivial ideal $\I$ of $\A$ such that $F$ factors though $\pi_\I:\A\to \A/\I$. 
\end{defi}

Observe that any (not necessarily full) subcategory of a $k$-trivial category is itself $k$-trivial. Thus, $F:\A\to k\Md$ is antipolynomial if and only if it factors through a $k$-trivial category. For example, let $\A$ be the category of finitely generated abelian groups and let $n$ be an integer invertible in $k$. Then for all $d$ the functor $F(a)=k[\Lambda^d(a/na)]$, sending an abelian group $a$ to the free $k$-module with basis the elements of the $d$-th exterior power of the abelian group $a/na$, is antipolynomial. Indeed, $F$ factors though the category of finitely generated $\mathbb{Z}/n\mathbb{Z}$-modules, which is $k$-trivial.

\begin{defi}\label{defi-AP}
A bifunctor $B:\A\times\A\to k\Md$ is of \emph{antipolynomial-polynomial type} (AP-type) if  for all objects $x$ of $\A$ the functor $y\mapsto B(y,x)$ is antipolynomial and the functor $y\mapsto B(x,y)$ is polynomial.
\end{defi}
Typical examples of bifunctors of AP-type are the bifunctors $B(x,y)=A(x)\otimes_k P(y)$ obtained by tensoring functors $A,P:\A\to k\Md$, where $A$ is antipolynomial and $P$ is polynomial.

The main result of the section is the following theorem, which is theorem \ref{thm-AP-type} of the introduction.

\begin{thm}[separation]\label{thm-AP-type-bis}Let $k$ be a commutative ring and let $B$, $C$, $B'$ be three bifunctors of AP-type, with $B'$ contravariant in both variables. Restriction along the diagonal $\Delta:\A\to \A\times\A$ yields graded isomorphisms:
\begin{align*}
&\res^\Delta\;:\;\Ext^\bullet_{k[\A\times\A]}(B,C)\simeq \Ext^\bullet_{k[\A]}(\Delta^*B,\Delta^*C)\;,\\
&\res_\Delta\;:\;\Tor_\bullet^{k[\A]}(\Delta^*B',\Delta^*C)\simeq \Tor_\bullet^{k[\A\times \A]}(B',C)\;.
\end{align*}
\end{thm}

\subsection{Resolutions of bifunctors of AP-type}
Given an ideal $\I$ of $\A$ and a positive integer $d$, we denote by $\C_{\I,d}$ the full subcategory of $k[\A\times\A]\Md$ whose objects are the bifunctors $B$ such that:
\begin{enumerate}[i)]
\item $B$ factors through $\pi_\I\times\mathrm{id}:\A\times\A\to (\A/\I)\times \A$, and
\item for all $x$, the functor $y\mapsto B(x,y)$ is polynomial of degree less than $d$. 
\end{enumerate}
The subcategory $\C_{\I,d}$ of $k[\A\times\A]\Md$ is stable under limits and colimits. Stability under colimits ensures that any object $B$ of $k[\A\times\A]\Md$ has a largest subobject $B_{\I,d}$ belonging to $\C_{\I,d}$. 

\begin{lm}\label{lm-exhaust}
If $B$ is a bifunctor of AP-type then 
$B=\bigcup_{\I,d} B_{\I,d}$, where $\I$ runs over the set of $k$-cotrivial ideals of $\A$ and $d$ runs over the set of positive integers. 
\end{lm}
\begin{proof}
We fix two objects $x,y$ of $\A$. Let $d$ be the degree of $t\mapsto B(x,t)$ and let $\I$ be a $k$-cotrivial ideal such that $s\mapsto B(s,y)$ factors through $\A/\I$. To prove the lemma, it suffices to show that the inclusion $B_{\I,d}\hookrightarrow B$ induces an equality $B_{\I,d}(x,y)=B(x,y)$.

Let $B_{d}(a,-)$ be the largest subfunctor of $B(a,-)$ of degree less than $d$. Any map $f:a\to b$ induces a map $B_{d}(a,-)\to B_{d}(b,-)$, so that the functors $B_{d}(a,-)$ assemble into a bifunctor $B_{d}:\A\times\A\to k\Md$ which is a subfunctor of $B$, polynomial of degree less of equal to $d$ with respect to its first variable. By construction $B_{d}(x,y)=B(x,y)$. 
Similarly, let $B_\I(-,b)$ be the largest subfunctor of $B(-,b)$ factorizing through $\A/\I$. These functors assemble into a bifunctor $B_\I:\A\times\A\to k\Md$  factorizing through $\A/\I\times\A$. By construction $B_\I(x,y)=B(x,y)$. 
Since $B_{\I}\cap B_d\subset B_{\I,d}$ we finally obtain that $B_{\I,d}(x,y)=B(x,y)$.
\end{proof}

An object $B$ of $k[\A\times\A]\Md$ is \emph{of special-AP-type} if there is an object $z$ of $\A$, a $k$-cotrivial ideal $\I$ and a polynomial functor $F$ in $k[\A]\Md$ such that 
\[B(x,y)= k[\A/\I(z,x)]\otimes F(y)\;.\]

\begin{lm}\label{lm-resolution}
Every bifunctor $B$ of AP-type has a resolution $Q\to B$ whose terms are direct sums of bifunctors of special-AP-type.
\end{lm}
\begin{proof}
It suffices to prove that every $B_{\I,d}$, is a quotient of a direct sum of bifunctors of special-AP-type. By definition $B_{\I,d}=(\pi_\I\times\mathrm{id}_\A)^*B'$ for some bifunctor $B':\A/\I\times\A\to k\Md$ such that each $B'_z(-):=B'(z,-)$ is polynomial of degree less than $d$. The standard resolution of $B'$ \cite[section 17]{Mi72} yields an epimorphism $\bigoplus_{z} h_{k[\A/\I]}^z\boxtimes B'_z\to B'$, where the sum is indexed by a set of representatives $z$ of isomorphism classes of objects of $\A/\I$. The result follows by restricting this epimorphism along $\pi_\I\times\mathrm{id}_\A$.
\end{proof}

\subsection{Proof of the separation theorem \ref{thm-AP-type-bis}} The proof will rely on the following vanishing lemma.

\begin{lm}\label{lm-annul}
Let $A$, $B$, and $F$ be three objects of $k[\A]\Md$, with $A$ reduced.
If one object of the pair $\{A,B\}$ is polynomial and the other is antipolynomial, then $\Ext^*_{k[\A]}(A\otimes F,B)= 0$.
\end{lm}

\begin{proof}
Let $A\boxtimes F$ denote the bifunctor such that $(A\boxtimes F)(x,y)=A(x)\otimes F(y)$. 
By proposition \ref{pr-sum-diagonal}, $\Ext^*_{k[\A]}(A\otimes F,B)$ is isomorphic to 
$\Ext^*_{k[\A\times\A]}(A\boxtimes F,\Sigma^*B)$. Hence if we let $B_y(x)=B(x\oplus y)$, then by proposition \ref{pr-no-Ext},  it suffices to show that for all $y$, we have $\Ext^*_{k[\A]}(A,B_y)= 0$. Note that $B_y$ is polynomial/antipolynomial if and only if $B$ is. Thus it suffices to show that when $A$ is reduced,
$$\Ext^*_{k[\A]}(A,B)=0\qquad(*)$$
whenever one argument of the $\Ext$ is polynomial and the other is antipolynomial.
There are two cases. 

{\bf Case 1: $A$ antipolynomial and $B$ polynomial.} Since $A$ is reduced antipolynomial, there is a $k$-cotrivial ideal $\I$ and a reduced functor $A':\A/\I\to k\Md$ such that $A=\pi_\I^*A'$. The reduced functor $A'$ has a projective resolution by direct sums of reduced standard projectives, hence $A$ has a resolution by direct sums of functors of the form 
$k[\A/\I(z,-)]^\red$ for some $z$ in $\A$. A $\delta$-functor argument then shows that in order to prove the cancellation $(*)$, it suffices to prove it for the functors $k[\A/\I(z,-)]^\red$, which follows from lemma \ref{lm-vanish-pol}.

{\bf Case 2: $A$ polynomial and $B$ antipolynomial.} There is a $k$-cotrivial ideal $\I$ and a functor $B':\A/\I\to k\Md$ such that $B=\pi_\I^*B'$. The functor $B'$ has an injective coresolution by products of standard injectives, hence $B$ has a coresolution by products of functors of the form $D_MQ$ with $Q=k[\A/\I(-,z)]$ and $M$ an injective $k$-module. A $\delta$-functor argument shows that it suffices to prove the cancellation $(*)$ when $B=D_MQ$ with $Q=k[\A/\I(-,z)]$. Proposition \ref{pr-memechose} yields a graded isomorphism:
$$\Ext^*_{k[\A]}(A,D_MQ)\simeq \Ext^*_{k[\A^{\op}]}(Q,D_MA)$$
and since $A$ is reduced, $D_MA$ is reduced and the right hand side of the isomorphism is equal to $\Ext^*_{k[\A^{\op}]}(Q^\red,D_MA)$. Now $Q^\red$ is a reduced antipolynomial functor and $D_MA$ is a polynomial functor, hence the latter $\Ext$ is zero as a consequence of the cancellation proved in case 1.
\end{proof}

\begin{proof}[Proof of theorem \ref{thm-AP-type}]
If $B$ is a bifunctor of AP-type, then its dual $D_MB$ is also a bifunctor of AP-type. Therefore, by proposition \ref{pr-memechose} it suffices to prove the $\Ext$-isomorphism of theorem \ref{thm-AP-type}. 
Moreover, lemma \ref{lm-resolution} provides a resolution $Q\to B$ of $B$ whose terms $Q_p$ are direct sums of bifunctors of special AP-type. Therefore, a standard $\delta$-functor argument shows that it suffices to prove theorem \ref{thm-AP-type} when $B$ is a bifunctor of special AP-type.

So we now assume that $B(x,y)=A(x)\otimes F(y)$ with $F$ polynomial of degree less than $d$ and $A=k[\A/\I(z,-)]$ for a $k$-cotrivial ideal $\I$ and an object $z$ of $\A$. 
We have a commutative diagram
\[
\begin{tikzcd}[column sep=large]
\Ext^\bullet_{k[\A\times\A]}(B,C)\ar{d}[swap]{\Ext^\bullet_{k[\A\times\A]}(\pi,C)}\ar{r}{\res^\Delta} & \Ext^\bullet_{k[\A]}(A\otimes F,\Delta^*C)\\
\Ext^\bullet_{k[\A\times\A]}(\Sigma^*(A\otimes F),C)\ar{ur}{\alpha}[swap]{\simeq}
\end{tikzcd}
\]
where $\alpha$ is the isomorphism of proposition \ref{pr-sum-diagonal} and $\pi:\Sigma^*(A\otimes F)\to B$ is the epimorphism of bifunctors induced by the canonical projections of $p_x:x\oplus y\to x$ and  $p_y:x\oplus y\to y$:
$$\pi_{x,y}:=A(p_x)\otimes F(p_y): A(x\oplus y)\otimes F(x\oplus y)\to A(x)\otimes F(y)\;.$$
Hence, in order to prove that $\res^\Delta$ is an isomorphism, it suffices to prove that $\Ext^\bullet_{k[\A\times\A]}(\pi,C)$ is an isomorphism. Now we use the natural decompositions:
\begin{align*}
A(x\oplus y)=A(x)\oplus A(x)\otimes A^{\red}(y)\;,&& F(x\oplus y)=F(x)\oplus \underbrace{F(y)\oplus \Cr_2F(x,y)}_{=G(x,y)}\;.
\end{align*}
They induce a decomposition
$\Sigma^*(A\otimes F)\simeq  B \oplus X \oplus Y$
where the bifunctors $X$ and $Y$ are defined by:
\begin{align*}
&X(x,y):= A(y)\otimes G(x,y)\;,&&Y(x,y):= A(x)\otimes A^\red(y)\otimes F(x\oplus y)\;.
\end{align*}
Moreover, the map $\pi$ identifies though this decomposition with the canonical projection on the summand $B$. Thus, in order to prove that $\res^\Delta$ is an isomorphism, it suffices to prove that there is no nonzero $\Ext$ between $X\oplus Y$ and $C$.  

But for all objects $y$ and $y'$, there is no nonzero $\Ext$ between the functors $X(-,y)$ and $C(-,y')$ by lemma \ref{lm-annul} since $C(-,y')$ is antipolynomial and $X(-,y)$ is reduced and polynomial. Hence there is no nonzero $\Ext$ between $X$ and $C$ by proposition \ref{pr-no-Ext}. Similarly, there is no nonzero $\Ext$ between $Y$ and $C$,  whence the result.
\end{proof}

\section{Appendix: simplicial techniques}\label{sec-app}
In this appendix, we gather the basic definitions and the results on simplicial objects that we need in the proof of the excision theorem \ref{thm-magique-general}. The reader may consult \cite[Chap 8]{Weibel} and \cite{GoerssJardine} as general references on simplicial objects. Besides our desire to provide the reader an easy access to these techniques, the main raison d'\^etre of this appendix is that we have not been able to find a reference in the literature proving the local Hurewicz theorem for arbitrary rings $k$ (the classical algebraic topology litterature being essentially focused on the cases where $k$ is the ring of integers or a prime field).
\subsection{Simplicial resolutions in an abelian category}
Recall that if $Q$ is a simplicial object in an abelian category $\M$, its \emph{homotopy groups} $\pi_*Q$ are the homology groups of the chain complex associated to $Q$. (If $\M=k\Md$, it can be proved that these homotopy groups coincide with the homotopy groups of the underlying simplicial set of $Q$).
We can view every object $X$ of $\M$ as a constant simplicial object of $X$. A \emph{simplicial resolution} of $X$ is a morphism of simplicial objects $Q\to X$ which induces an isomorphism on the level of homotopy groups. The simplicial resolution is called projective if the simplicial object $Q$ is degreewise projective. The Dold-Kan equivalence \cite[Section 8.4]{Weibel} ensures that if $\M$ has enough projectives, then every object has a simplicial projective resolution.

\subsection{Eilenberg MacLane spaces and Hurewicz theorems}
For all abelian groups $A$ and all $n\ge 0$, we denote by $K(A,n)$ any simplicial free abelian group such that $\pi_iK(A,n)=0$ for $i\ne n$ and $\pi_n K(A,n)\simeq A$. Such a simplicial free abelian group is called an \emph{Eilenberg-MacLane space} and is unique up to homotopy equivalence. The study of simplicial abelian groups often reduces to that of Eilenberg-MacLane spaces by the following classical lemma, see e.g. \cite[Prop 2.20]{GoerssJardine}. We impose that Eilenberg-MacLane spaces are degreewise free abelian groups by definition in order to have genuine maps rather than zig-zags in this lemma.
\begin{lm}\label{lm-decomp}
For all simplicial abelian groups $X$, there is a weak equivalence (unique up to homotopy)
\[\prod_{i\ge 0} K(\pi_i X,i)\xrightarrow{\simeq} X\;.\]
Moreover for all morphisms of simplicial abelian groups $f:X\to Y$, let $K(\pi_i f,i): K(\pi_i X,i)\to K(\pi_i Y,i)$ denote a lift of $\pi_i f: \pi_iX\to \pi_i Y$ to the level of the simplicial resolutions. Then the following diagram commutes up to homotopy:
\[
\begin{tikzcd}[column sep=huge]
\prod_{i\ge 0} K(\pi_i X,i)\ar{d}\ar{r}{\prod K(\pi_i f,i)}& \prod_{i\ge 0} K(\pi_i Y,i)\ar{d}\\
X\ar{r}{f}& Y
\end{tikzcd}\;.
\]
\end{lm}

If $X$ is a simplicial abelian group, the inclusion of simplicial sets $X\to \mathbb{Z}[X]$ induces a natural morphism of graded abelian groups
\[h_*:\pi_*X\to \pi_*\mathbb{Z}[X]\]
called the \emph{Hurewicz morphism}, and whose properties are described in the Hurewicz theorem. 
We recall that a simplicial abelian group $X$ is called \textit{$e$-connected} if $\pi_iX=0$ for $i\le e$, and that a morphism of simplicial abelian groups $f:X\to Y$ is $e$-connected if $\pi_i(f):\pi_iX\to \pi_iY$ is an isomorphism for $i<e$ and an epimorphism for $i=e$.  
\begin{pr}[Classical Hurewicz Theorems]\label{pr-Hur-classical}Let $e$ be a nonnegative integer. 
\begin{enumerate}
\item[(1)] (Absolute theorem)
For all simplicial abelian groups $X$,
the Hurewicz map $h_*$ is split injective. Moreover, if $X$ is $e$-connected then $h_i$ is an isomorphism for $i\le e+1$. 
\item[(2)] (Relative theorem)
Every $e$-connected morphism of simplicial abelian groups $f:X\to Y$ induces an $e$-connected morphism of simplicial $k$-modules $k[f]:k[X]\to k[Y]$ for all commutative rings $k$. 
\end{enumerate}
\end{pr}
\begin{proof}
(1) The canonical morphism of abelian groups $\mathbb{Z}[X]\to X$ yields a retract of $h_*$. The isomorphism is given by \cite[III Thm 3.7]{GoerssJardine}.

(2) By the universal coefficient theorem, it suffices to prove the result for $k=\mathbb{Z}$. Since simplicial groups are fibrant simplicial sets \cite[I Lm 3.4]{GoerssJardine}, any weak equivalence between simplicial groups yields a homotopy equivalence of simplicial sets \cite[II Thm 1.10]{GoerssJardine}, hence it induces an isomorphism in homology. Therefore, lemma \ref{lm-decomp} and the K\"unneth theorem reduce the proof of the isomorphism to the case where $X$ and $Y$ are Eilenberg-MacLane spaces, with nonzero homotopy groups placed in the same degree $i$. If $i<e$, $f$ is $e$-connected if and only if it is a weak equivalence, hence if and only if it induces an isomorphism in homology. If $i\ge e$, then $A$ and $B$ are $e$-connected hence the result follows from (1).  
\end{proof}

For our purposes, we need a $k$-local version of the absolute Hurewicz theorem. We shall derive it from the following well-known property of Eilenberg-MacLane spaces, which we have not found in the literature -- though the case of a prime field $k$ is of course given by the  classical calculations of Cartan \cite{Cartan}. 

\begin{lm}\label{lm-EML-vanish}
Let $k$ be a commutative ring, let $A$ be an abelian group. If $k\otimes_\mathbb{Z}A=0$ and $\Tor_1^{\mathbb{Z}}(k,A)=0$, then $\pi_ik[K(A,n)]=0$ for all positive integers $n$ and $i$.
\end{lm}
\begin{proof}
We say that an abelian group $A$ is \textit{$k$-negligible} if $\Tor_1^{\mathbb{Z}}(k,A)=0=k\otimes_{\mathbb{Z}}A$. 

We first take $n=1$. Then $\pi_*k[K(A,1)]$ is the homology with coefficients $k$ of the abelian group $A$. We start from the well-known natural isomorphism, valid for all free abelian groups $A$:
$$\pi_*k[K(A,1)]=k\otimes\Lambda_\mathbb{Z}^*(A)=\Lambda_k^*(k\otimes_{\mathbb{Z}}A)\;.$$
(When $k=\mathbb{Z}$ this is proved e.g. in \cite{Brown}, and it extends to an arbitrary $k$ by the uiversal coefficient theorem). Since every torsion-free abelian group is the filtered colimit of its finitely generated free groups, this isomorphism extends to all torsion-free abelian groups $A$ by taking colimits. In particular, $\pi_{>0}k[K(A,1)]=0$
if $A$ is a $k$-negligible torsion-free group.
We claim that we also have $\pi_{>0}k[K(A,1)]=0$ for all $k$-negligible torsion group $A$. Indeed, $A$ is the filtered colimit of its subgroups $_nA$ of $n$-torsion elements, and these subgroups are $k$-negligible. So it suffices to prove the result when $A$ is a $k$-negligible group of bounded order. As such groups are direct sums of cyclic groups, the proof reduces further to the case of the $k$-negligible cyclic groups, and for the latter, the result follows by a direct computation. Now let $A$ be an arbitrary abelian group with torsion subgroup $A_{\mathrm{tors}}$. If $A$ is $k$-negligible, then so are $A_{\mathrm{tors}}$ and $A/A_{\mathrm{tors}}$. So the lemma holds for $A$ as a consequence of the Hochschild-Serre spectral sequence of the fibration $K(A_{\mathrm{tors}},1)\to K(A,1)\to K(A/A_{\mathrm{tors}},1)$.

Assume now that $n>1$. Then $K(A,1)\otimes_{\mathbb{Z}} K(\mathbb{Z},n-1)$ is an Eilenberg MacLane space $K(A,n)$. Thus $\pi_*k[K(A,n)]$ is the abutment of the spectral sequence of the bisimplicial $k$-module $M_{pq}= k[K(\mathbb{Z},n-1)_q\otimes_{\mathbb{Z}}K(\pi,1)_p]$. Let us choose $K(\mathbb{Z},n-1)$ such that it is free of finite rank $r(q)$ in each degree $q$ (e.g. take the image of the complex $\mathbb{Z}[-n]$ by the Kan functor). 
Then for $q$ fixed, the simplicial $k$-module $M_{\bullet q}$ is isomorphic to $k[K(A^{\times r(q)},1)]$. Thus the simplicial spectral sequence of $M_{pq}$ can be rewriten as:
\[E^1_{pq}=\pi_pk[K(A^{\times r(q)},1)]\Longrightarrow \pi_{p+q}k[K(A,n)]\;.\]
The first page is zero by the case $n=1$, whence the result.
\end{proof}

\begin{pr}[$k$-local absolute Hurewicz theorem]\label{pr-Hur}
Let $X$ be a simplicial abelian group, let $k$ be a commutative ring and let $e$ be a non-negative integer. Assume that for $0<i \le e$ and for $0<j<e$ we have $k\otimes_{\mathbb{Z}}\pi_iX=0=\mathrm{Tor}^{\mathbb{Z}}_1(k,\pi_jX)$. Then 
\begin{enumerate}
\item[(1)] $\pi_0k[X]=k[\pi_0 X]$;
\item[(2)] $\pi_ik[X]=0$ for $0<i\le e$;
\item[(3)] $\pi_{e+1}k[X]$ contains the following $k$-module as a direct summand: 
$$k[\pi_0 X]\otimes_{k}\left(k\otimes_{\mathbb{Z}}\pi_{e+1} X\,\oplus \,\mathrm{Tor}^{\mathbb{Z}}_1(k,\pi_{e}X)\right)\;.$$
\end{enumerate}
\end{pr}
\begin{proof}
Lemma \ref{lm-decomp} and the K\"unneth theorem reduce the proof to the case of an Eilenberg-MacLane space $X$. Assume that the nonzero homotopy group of $X$ is placed in degree $i$. If $i=0$, the result holds by a direct computation. If $0<i<e$ then the result follows from lemma \ref{lm-EML-vanish}. If $i\ge e$, the result follows from the classical absolute Hurewicz theorem of proposition \ref{pr-Hur-classical} together with the universal coefficient theorem which says that the graded $k$-module 
$\pi_*k[X]$ is (non canonically) isomorphic to $k\otimes_{\mathbb{Z}} \pi_*\mathbb{Z}[X]\oplus 
\Tor^{\mathbb{Z}}_1(k,\pi_{*-1}\mathbb{Z}[X])$.
\end{proof}
\begin{cor}\label{cor-Hur}
The $k$-modules $\pi_ik[X]$ vanish for $0<i\le e$ if and only if $k\otimes_{\mathbb{Z}}\pi_iX$ and $\Tor^\mathbb{Z}_1(k,\pi_jX)$ vanish for $0<i\le e$ and $0<j<e$. 
\end{cor}

\bibliographystyle{amsplain}

\bibliography{biblio-DT-separationexcision}

\end{document}